\documentclass[11pt, reqno]{amsart}
\usepackage{preamble}

\title[Poset-enriched pretoposes and compact ordered spaces]{Poset-enriched pretoposes\\and compact ordered spaces}

\author{J\'er\'emie Marqu\`es}
\author{Luca Reggio}
\address{Dipartimento di Matematica ``Federigo Enriques'', Universit\`a degli Studi di Milano, via Saldini 50, 20133 Milano, Italy}
\email{contact@jeremie-marques.name}
\email{luca.reggio@unimi.it}

\date{\today}

\bibliography{biblio}

\begin{document}

\maketitle

\begin{abstract}
We provide a characterisation of the category $\KOrd$ of Nachbin's compact ordered spaces as a poset-enriched category. Up to equivalence, $\KOrd$ is the only non-degenerate poset-enriched pretopos whose terminal object is a (discrete) generator and in which every object is covered by an \emph{order-filtral} object. Order-filtral objects satisfy an appropriate form of compactness and separation. Throughout, we make extensive use of the internal language of poset-enriched pretoposes.
\end{abstract}


\section{Introduction}

Compact ordered spaces, introduced by Nachbin in~\cite{Nachbin1965}, are compact topological spaces equipped with a partial order that is closed in the product topology. The latter condition ensures compatibility between the topology and the order, and implies that every compact ordered space is Hausdorff. The natural morphisms between compact ordered spaces are the continuous order-preserving maps. We denote by $\KOrd$ the resulting category. 
Every compact Hausdorff space can be viewed as a compact ordered space with the trivial order, which gives rise to a fully faithful functor $\KH \into \KOrd$, from the category $\KH$ of compact Hausdorff spaces and continuous maps, into $\KOrd$. 
In logic, compact ordered spaces naturally arise as spaces of (model-theoretic) types of theories in coherent first-order logic, and extensions thereof; see e.g.~\cite{vGM2024}.

In~\cite{mr}, the category $\KH$ was characterised, up to equivalence, as the unique non-degenerate pretopos $\C$ such that its terminal object is a generator (equivalently, $\C$ is well-pointed and admits copowers of its terminal object), and each of its objects is covered by a \emph{filtral} object. An object of a coherent category is filtral if its lattice of subobjects is canonically isomorphic to the filter completion of its Boolean center \cite[Definition~4.1]{mr}. In $\KH$, filtral objects coincide with Stone spaces, i.e.\ totally disconnected compact Hausdorff spaces. The property of having enough filtral objects distinguishes $\KH$ from the category $\Set$ of sets and functions. The latter is also a non-degenerate pretopos whose terminal object is a generator; however, the filtral objects in $\Set$ are the finite sets, therefore no infinite set is covered by a filtral object.

The main result of this paper is a characterisation, up to equivalence, of the category $\KOrd$ of compact ordered spaces. Our characterisation takes place within the class of \emph{poset-enriched categories}, i.e., categories enriched in the category $\Pos$ of posets and order-preserving maps. Any ordinary category can be regarded as a poset-enriched category, albeit not necessarily in a unique way. This means that we are considering a broader setting. However, the advantage is that we can leverage a richer language. In fact, given a poset-enriched category with finite limits, its ordinary internal language is naturally enriched with a partial order $\leq_{X}$ for each object $X$, and morphisms respect these orders. This viewpoint will be developed in detail in Section~\ref{s:internal-logic}. 

Our main result can be stated as follows (for all undefined terms, see the additional details and references below): 

\begin{theorem}\label{thm:main-KOrd}
A poset-enriched category $\C$ is equivalent to $\KOrd$ if, and only if, it satisfies the following conditions:
\begin{enumerate}[label=(\roman*)]
\item $\C$ is a non-degenerate poset-enriched pretopos;
\item The terminal object of $\C$ is a discrete generator;
\item Every object of $\C$ is covered by an order-filtral object.
\end{enumerate}
\end{theorem}
While the latter characterisation of $\KOrd$ is formally very similar to that of $\KH$ obtained in~\cite{mr}, the key difference lies in the poset-enrichment. Indeed, while $\KOrd$ is not a pretopos in the ordinary sense, $\KH$ is not a poset-enriched pretopos when considered as a poset-enriched category in the obvious way.\footnote{This is ultimately because the notion of congruence in the poset-enriched setting does not coincide with the ordinary one, even for categories with trivial poset-enrichment.} Basic facts concerning poset-enriched categories will be reviewed in Section~\ref{s:preliminaries}.

In Section~\ref{s:internal-logic}, we discuss poset-enriched pretoposes, a simpler version of the notion of \emph{$2$-pretopos} in $2$-category theory. A $2$-pretopos is a (strict) $2$-category that is exact and has universal, disjoint finite coproducts, in an appropriate $2$-categorical sense, as defined by Street in~\cite{Street1982}. A $2$-category in which all pairs of parallel $2$-morphisms are equal is called a \emph{$(1,2)$-category}. While $2$-categories are categories enriched in $\cat{Cat}$, $(1,2)$-categories are categories enriched in the category of preorders and order-preserving maps. By adapting the notion of a $2$-pretopos to this setting, we define \emph{poset-enriched pretoposes}, which can also be seen as $(1,2)$-pretoposes. The $2$-categorical exactness conditions will play a key role in the following. However, we shall only present them in elementary terms for poset-enriched categories, using the internal language.

The notion of a generator in Theorem~\ref{thm:main-KOrd}, namely a \emph{discrete generator}, is weaker than the more standard notion of generator in poset-enriched categories (as used in, e.g.,~\cite{KV2017}). The main difference lies in the use of copowers rather than arbitrary tensors. However, the two notions coincide for poset-enriched exact categories and, a fortiori, for poset-enriched pretoposes. Generators and projective covers are discussed in Section~\ref{s:projectives-and-generators}, where we present a poset-enriched variant of Barr's embedding theorem for poset-enriched regular categories with enough projectives. In particular, if such a category is exact, it can be recovered from any of its projective covers. As an aside, we derive a sufficient criterion for a poset-enriched exact category to be monadic over~$\Pos$.

In Section~\ref{s:well-pointed-pretoposes}, we study poset-enriched pretoposes $\cC$ whose terminal object is a generator. These are precisely the well-pointed cocomplete pretoposes and admit a projective cover consisting of the copowers of their terminal object.
Finally, in Section~\ref{s:char-KOrd}, we introduce the concept of an \emph{order-filtral} object (Definition~\ref{d:order-filtral}), generalising filtrality to the poset-enriched setting. Order-filtral objects in $\KOrd$ are the \emph{Priestley spaces}, i.e., the totally order-disconnected compact ordered spaces~\cite{Priestley1970}. In $\Pos$, they coincide with the finite posets. Order-filtral objects are in particular \emph{compact} and \emph{separated}. If $\C$ has enough order-filtral objects, we show that $\C$ and $\KOrd$ admit equivalent projective covers, and are therefore equivalent. This leads to a proof of Theorem~\ref{thm:main-KOrd}.

\subsection*{Related work}
As an abstract characterisation of a category of ``spaces'', our main result is reminiscent of, and inspired by, Lawvere's Elementary Theory of the Category of Sets~\cite{Lawvere1964,Lawvere2005}. For a $2$-categorical treatment, see the recent paper~\cite{HM2025}. The category of compact Hausdorff spaces has been characterised in~\cite{Richter1991,Richter1992} and, using the pretopos structure, in~\cite{mr}; see also~\cite{BKRT2024} for a constructive approach.

Our work is also related to the study of varieties of ordered algebras, started in~\cite{Bloom1976,BW83}, and of (strongly finitary) monads on $\Pos$. See e.g.\ \cite{KV2017,AR2023,Adamek2025}. Our approach hinges on the use of the internal language of poset-enriched pretoposes. Enriched regular fragments have been recently investigated in a general setting in~\cite{RT25}. 

\subsection*{Acknowledgements}
The second-named author would like to thank Richard Garner for the email discussions that inspired this paper.

\subsection*{Notations} Throughout this paper, the composition of arrows $f\colon X\to Y$ and $g\colon Y\to Z$ in a category is denoted by $fg$. Furthermore, if $P$ is a poset, we write $\abs{P}$ for its underlying set. Whenever convenient, sets are identified with discrete posets.

\section{Preliminaries on poset-enriched categories}\label{s:preliminaries}

A \emph{poset-enriched category} is a category enriched in the category $\Pos$ of posets and order-preserving maps. In other words, it is a category $\cA$ whose hom-sets are equipped with partial orders in such a way that composition is order-preserving. For any two objects $x,y\in \cA$, the poset of arrows $x\to y$ is denoted by $\cA(x,y)$. 
Many of the basic concepts of ordinary category theory extend to enriched categories, and in particular to poset-enriched categories. We refer the reader to \cite{Kelly2005reprint} for a thorough treatment.

A \emph{poset-enriched functor} is a functor $F \colon \cA \to \cB$ between poset-enriched categories such that the maps $\cA(x,y) \to \cB(F(x),F(y))$ are order-preserving. We say that $F$ is \emph{faithful} if these maps are order-embeddings, and \emph{fully faithful} if they are order-isomorphisms. Given two poset-enriched categories $\cA$ and $\cB$, the category of poset-enriched functors $[\cA,\cB]$ is also poset-enriched (although it is not locally small unless $\cA$ is small).

By a \emph{category} or a \emph{functor}, we will always mean a poset-enriched category or a poset-enriched functor. An \emph{ordinary} category or functor is one that is not enriched. We may occasionally still speak of a \emph{poset-enriched} category or functor to emphasise that it is not ordinary. The canonical way to regard a poset-enriched category as an ordinary category is to discard the orders. Conversely, the canonical way to view an ordinary category as a poset-enriched category is to equip its hom-sets with discrete orders.

Given a category $\cA$, the category $\cA^\op$ is obtained by reversing the direction of the arrows, but \emph{not} their order. A functor $\cA^\op \to \Pos$ is called a (poset-enriched) \emph{presheaf}. The hom-functor $\cA^\op \times \cA \to \Pos$ is poset-enriched, and by currying, we obtain the enriched \emph{Yoneda embedding} $\cA \to [\cA^\op,\Pos]$. As in the ordinary case, the Yoneda embedding is fully faithful; see, for example, \cite[\S2.4]{Kelly2005reprint}.

The notions of limit and colimit for poset-enriched categories are direct extensions of the corresponding notions for ordinary categories, yet they do not reduce to them. For example, a product of objects $X$ and $Y$ in a poset-enriched category $\cC$ is an object $X\times Y$ such that, for every object $Z$, the natural maps
\[ \C(Z,X\times Y) \to \C(Z,X) \times \C(Z,Y) \]
are order-isomorphisms. Thus, while $\cC$ may have binary products as an ordinary category, it may not have them as a poset-enriched category because the natural bijections may fail to reflect the order. In the poset-enriched setting, limits and colimits are defined by requiring the associated natural bijections to be order-isomorphisms.

In a poset-enriched category $\cC$, the usual limits, built from equalisers and products, are called \emph{conical} and are a special case of \emph{weighted} limits. Dually, conical colimits are a special case of weighted colimits. We recall some examples of non-conical limits and colimits that will be relevant later on, by describing the presheaves that they represent:
\begin{itemize}
	\item The \emph{lax pullback}\footnote{Lax pullbacks are often referred to as \emph{comma squares} in the literature. The dual notion of \emph{lax pushout} is often called \emph{cocomma square}.} of arrows $f \colon X \to Y$ and $g \colon Z \to Y$ represents the pairs of arrows $(u \colon T \to X, v \colon T \to Z)$ such that $uf \leq vg$.
	\item The \emph{lax kernel} of $f \colon X \to Y$ is the lax pullback of $f$ along itself, and is denoted by $\lker(f)$. It represents the pairs of arrows $u,v \colon T \to X$ such that $uf \leq vf$. 
	\item The \emph{coinserter} (or ``lax coequaliser'') of arrows $f,g \colon X \rightrightarrows Y$ represents the arrows $u \colon Y \to T$ such that $fu \leq gu$.
	\item Given a poset $P$ and an object $X$, the \emph{tensor} $P\cdot X$ satisfies the universal property
	\[ \forall Y\in \cC, \ \ \cC(P\cdot X,Y) \cong \Pos(P,\cC(X,Y)) . \]
	If $S$ is a set, we shall refer to $S\cdot X$ as a \emph{copower} of $X$ indexed by $S$.
\end{itemize}

General weighted limits are obtained as follows. We start with a diagram $F \colon \cI \to \cC$ where $\cI$ is an ordinary category (in the general theory of enriched categories, $\cI$ is also enriched, but this is not necessary in our context). The \emph{(conical) limit} of $F$ represents the families of arrows $T \to F(i)$, with exactly one arrow for each $i \in \cI$, such that for each $f \colon i \to j$ in $\cI$, the arrow $F(f)$ sends $T \to F(i)$ to $T \to F(j)$. In a weighted limit, the single arrow $T \to F(i)$ is replaced by a family of arrows $T \to F(i)$ indexed by a given poset, and we specify how each $F(f)$ acts on these arrows. Formally, this data is specified by giving a \emph{weight} functor $W \colon \cI \to \Pos$, where $W(i)$ is the indexing poset of the arrows $T \to F(i)$. The \emph{weighted limit} $\lim\nolimits^W F$ satisfies the following universal property, where $\operatorname{\mathsf{Nat}}_i$ denotes the transformations natural in $i$:\footnote{That is, $\operatorname{\mathsf{Nat}}_i(W(i),\cC(T,F(i)))\coloneqq [\cI, \Pos](W(-), \cC(T,F(-)))$.}
\[ \cC(T,\lim\nolimits^W F) \cong \operatorname{\mathsf{Nat}}_i(W(i),\cC(T,F(i))). \]

Below, we will make use of the notion of a \emph{finite} weighted limit. This requires that each weight $W(i)$ be a finite poset and that the indexing category $\cI$ be \emph{finitely presented}, meaning it admits a presentation in terms of finitely many objects, finitely many arrows, and finitely many equations. Note that a finitely presented category may have infinitely many arrows. Alternatively, we can obtain the concept of a finite weighted limit by replacing the category~$\cI$ with a finite graph in the above definition. In Remark~\ref{rmk:weighted-lim-internal}, we will describe finite weighted limits using the internal language.

\section{Poset-enriched pretoposes and their internal logic}\label{s:internal-logic}

In this section, we recall the notion of poset-enriched pretopos. As mentioned in the Introduction, this is a simpler version of the concept of $2$-pretopos, i.e.\ an exact $2$-category with universal, disjoint finite coproducts. These $2$-categorical exactness conditions were introduced in~\cite{Street1982}; for the poset-enriched case, see~\cite{KV2017,AV2022}. 

Rather than recalling the general $2$-categorical definition, we will introduce poset-enriched pretoposes and their internal language in an elementary way. This approach is guided by the analogy with ordinary pretoposes, which are syntactic categories of first-order coherent theories. First, we will review the internal language of an ordinary category with finite limits. For a poset-enriched category with finite limits, we will see that its (ordinary) internal language is further equipped with a distinguished partial order for each object and that morphisms preserve these orders. By progressively adding more structure, we arrive at the notion of a poset-enriched pretopos. At each step, we will explain how the internal language is enriched or modified.

\subsection{Finite limits}
\label{subsec:fin-lim}

We briefly recall the internal logic of categories with finite limits, that is, \emph{finite-limit} or \emph{cartesian} logic. For more details, the reader can consult \cite[Chapter~D1]{Elephant2}. Formulas of cartesian logic for a given first-order signature are formed in the usual way by using the logical operations of conjunction $\phi \land \psi$ and the equality of variables $x=y$, in addition to the truth constant $\top$.\footnote{For simplicity, we leave aside provably unique existential quantification, which is needed to obtain an exact correspondence between formulas and subobjects when constructing the syntactic category of a cartesian theory. However, this is not necessary here.}

Let $\C$ be an ordinary category with finite limits. The \emph{internal language} of $\C$ consists of a multi-sorted first-order signature, together with a canonical interpretation in $\C$ of cartesian formulas on that signature. The signature is defined as follows:
\begin{itemize}
	\item the sorts are the objects of $\C$;
	\item the predicates of type $X_1 \times \cdots \times X_n$ are the subobjects of $X_1 \times \cdots \times X_n$ in $\C$, i.e., the equivalence classes of monomorphisms with this codomain;
	\item the operations of type $X_1 \times \cdots \times X_n \to Y$ are the morphisms in $\C$ with matching domain and codomain.
\end{itemize}
Each term $t$ in this signature comes equipped with a \emph{context}, which is a finite family of typed variables $(x_i \tcolon X_i)_i$ that contains all the variables appearing in $t$. Similarly, each cartesian formula $\phi$ comes with a context containing all the variables that occur freely in $\phi$. However, not every variable in the respective context needs to occur in $t$, or freely in $\phi$. Given a term $t$ in context $(x_i\tcolon X_i)_i$, its interpretation $\llbracket t\rrbracket \colon \prod_i X_i \to Y$ is defined inductively: $\llbracket x_j\rrbracket \colon \prod_i X_i \to X_j$ is the canonical projection, and $\llbracket f(t_1,\dots,t_n) \rrbracket$ is the composite of $\llbracket t_1\rrbracket \times \dots \times \llbracket t_n\rrbracket$ and $\llbracket f\rrbracket$. Given a formula $\phi$ in context $(x_i\tcolon X_i)_i$, its interpretation $\llbracket \phi\rrbracket$ is the subobject of $\prod_i X_i$ defined inductively as follows:
\begin{itemize}
	\item when $R$ is a predicate, $\llbracket R(t_1,\dots,t_m) \rrbracket$ is the pullback of $R$ by $\llbracket t_1\rrbracket \times \dots \times \llbracket t_m\rrbracket$, where $t_1,\dots,t_m$ is a list of terms in context $(x_i \tcolon X_i)_i$;
	\item conjunctions are interpreted as pullbacks (i.e., as intersections of subobjects);
	\item equality is interpreted as the diagonal $\langle\id,\id\rangle \colon X \into X\times X$, written $\Delta_X$.
\end{itemize}

Given two formulas $\phi$ and $\psi$ in the same context, we write 
\[
\phi \proves \psi
\] 
provided that $\llbracket \phi \rrbracket \leq \llbracket \psi \rrbracket$ as subobjects. We sometimes write $\proves \phi$ instead of $\top \proves\phi$. 

An important feature of the internal language is the following \emph{principle of compositionality}. Let $\phi$ be a formula in context $x_1\tcolon X_1,\dots,x_n\tcolon X_n$. Let $t_1,\dots,t_n$ be terms with common domain, and respective codomains $X_1,\dots,X_n$. We denote by $\phi[t_1/x_1,\dots,t_n/x_n]$ the formula obtained by simultaneously substituting $x_i$ by $t_i$ in $\phi$. 
The interpretation $\llbracket \phi \rrbracket$ of $\phi$ is a subobject of $X_1\times\cdots\times X_n$, so it can be seen as a predicate of this type, from which we can build the formula $\llbracket \phi \rrbracket(t_1,\dots,t_n)$. Then
\[ \llbracket \phi[t_1/x_1,\dots,t_n/x_n] \rrbracket = \llbracket \llbracket \phi \rrbracket(t_1,\dots,t_n) \rrbracket . \]
In other words, substitutions are realised as pullbacks. This can also be regarded as a principle of substitution of equals for equals: if two formulas are equivalent, they will remain so after a substitution.

\begin{example}
	We can define the concept of an (internal) equivalence relation in the internal language: it is a predicate $x \sim y$ of type $X \times X$ such that $x\sim y \proves y\sim x$, $\proves x\sim x$, and $(x\sim y) \land (y\sim z) \proves x\sim z$. A preorder relation $\leq_X$ on $X$ is defined similarly, but without the symmetry condition. A morphism $f \colon X \to Y$ is said to be \emph{order-preserving} with respect to two preorders $\leq_X$ and $\leq_Y$ if $a \leq_X b \proves f(a) \leq_Y f(b)$.
\end{example}

The internal language provides a concise notation for manipulating subobjects. In practice, we will describe the proofs in plain English where more convenient, manipulating the variables as if they were elements of sets. However, the reader should bear in mind that these proofs can be unpacked. We illustrate this ``unpacking'' with an example.

\begin{example}
Using the internal language, let us show that, in a category $\cC$ with finite limits, the composite of two order-preserving arrows $f \colon X\to Y$ and $g \colon Y\to Z$ is order-preserving. The formal proof is simply 
	\[
	a \leq_X b \proves f(a) \leq_Y f(b) \proves g(f(a)) \leq_Z g(f(b)).
	\] 
The first entailment, which is an inclusion of subobjects in $X\times X$, holds because $f$ is order-preserving. The second entailment is obtained from $u \leq_Y v \proves g(u) \leq_Z g(v)$ by substituting $u$ with $f(a)$, and $v$ with $f(b)$. To interpret this in $\cC$, we start with the inclusion of $[\leq_Y]$ in the pullback of $[\leq_Z]$ along $g\times g$. The substitution corresponds to pulling back once more along $f\times f$. As pullback squares compose, the right-hand side of the inclusion is the pullback of $[\leq_Z]$ along $(fg)\times(fg)$. In terms of the internal logic, the fact that these pullbacks compose is the aforementioned principle of compositionality. Finally, we obtain $a \leq_X b \proves g(f(a)) \leq_Z g(f(b))$ by composing the inclusions.
\end{example}

We will often use the set comprehension notation to make the context of a formula explicit. For instance, 
\[
\setst{(x,y) \in X\times Y}{\phi(x,y)}
\] 
stands for $\llbracket \phi \rrbracket$, where $\phi$ is understood as a formula in context $x\tcolon X, y\tcolon Y$, even though $x$ and $y$ do not necessarily occur freely in $\phi$. 

\begin{remark}\label{rmk:external-to-internal-regular}
All concepts expressible in cartesian logic admit an equivalent formulation using generalised elements (recall that a \emph{generalised element} of $X$ is an arrow $T\to X$). For instance, an internal equivalence relation is often defined as a subobject $R$ of $X\times X$ such that $\C(T,R) \subseteq \C(T,X) \times \C(T,X)$ is an equivalence relation, in the usual set-theoretic sense, for every object $T$. This correspondence between internal and external descriptions holds because the Yoneda embedding preserves (finite) limits and is conservative. For any formula $\phi(x)$, the subobject $\setst{p \in X}{\phi(p)}$ of $X$ represents the sub-presheaf of $\C(-,X)$ that takes $T$ to the set of arrows $p \colon T \to X$ such that $\phi(p)$ is true when interpreted as an ``external'' statement: conjunctions and equalities are interpreted externally, and $f(p)$ is interpreted as the composite $pf$.
\end{remark}

\subsection{Finite weighted limits} 

Firstly, let us review the interplay, in an ordinary category~$\cC$, between the external equality of arrows and the diagonals $\Delta_X\colon X \into X\times X$. The object $X\times X$ represents the pairs of parallel morphisms to $X$, and the diagonal represents the sub-presheaf of those pairs of morphisms that are \emph{equal}:
\[\begin{tikzcd}[row sep=1.6em]
	\C(Y,X\times X) &[-2em] \cong &[-2em] \set{ f,g \colon Y \rightrightarrows X } \\
	\C(Y,X) \ar[u,hook] & \cong & \setst{ f,g \colon Y \rightrightarrows X }{ f=g } \ar[u,hook]
\end{tikzcd}\]

Similarly, when $\C$ is poset-enriched, we require the partial order  
\[
\setst{ f,g \colon Y \rightrightarrows X }{ f \leq g } \subseteq \C(Y,X\times X)
\] 
to be representable by a subobject $[\leq_X] \into X\times X$.

\begin{definition}
	A category $\C$ with finite limits has \emph{epi-diagonals} if, for every $X\in \C$, the presheaf
	\begin{equation*}
		F_{X}\colon \C^{\op}\to\Pos, \ \ Y\mapsto \setst{f, g\colon Y \rightrightarrows X }{ f\leq g}
	\end{equation*}
	is representable. In this case, the object representing $F_X$ is denoted by $[\leq_X] \into X\times X$.
\end{definition}

Note that, if it exists, the epi-diagonal $[\leq_X] \into X\times X$ is a partial order by Remark~\ref{rmk:external-to-internal-regular}. 

We will now assume that $\cC$ has finite limits and epi-diagonals. In particular, it has finite limits as an ordinary category. Furthermore, its internal language as an ordinary category is enriched with a partial order $\leq_X$ for each object $X$. These orders are ``uniform'' in the sense that every morphism $f \colon X \to Y$ is order-preserving. Indeed, we have $x \leq_{X} x' \proves f(x) \leq_{Y} f(x')$, as can be seen by reasoning as in Remark~\ref{rmk:external-to-internal-regular}, replacing $x$ and $x'$ with two parallel arrows in $\cC$. We will simply write $\leq$ instead of $\leq_X$. 

\begin{definition}
A morphism $f \colon X \to Y$ in $\cC$ is an \emph{embedding} if $f(x) \leq f(y) \proves x\leq y$, and an \emph{injection} if $f(x) = f(y) \proves x = y$. Embeddings are denoted by $X\into Y$.
\end{definition}

By Remark~\ref{rmk:external-to-internal-regular}, $f$ is an embedding if and only if the induced map $\C(T,X) \to \C(T,Y)$ is an order-embedding for all $T$; that is, in $2$-categorical terminology, $f$ is an \emph{$\ff$-morphism}, where $\ff$ stands for \emph{(representably) fully faithful}. Similarly, $f$ is an injection precisely when the induced maps $\C(T,X) \to \C(T,Y)$ are injective, i.e., when $f$ is a monomorphism.

\begin{example}\label{ex:split-mono-emb}
	Any split monomorphism $f \colon A \to B$ is an embedding. Just observe that if $g \colon B \to A$ satisfies $fg=\id$, then $f(x) \leq f(y) \proves g(f(x)) \leq g(f(y)) \proves x \leq y$.
\end{example}

We will now impose an important restriction on the internal language of $\C$: rather than allowing the predicates to be arbitrary monomorphisms, we will require them to be embeddings. By structural induction, it can be shown that all formulas are then interpreted as embeddings because embeddings are stable under pullbacks (this can be seen easily using the ordinary internal language). Consequently, we make the following

\begin{definition}
A \emph{subobject} of $X\in \cC$ is an embedding $Y \into X$, modulo isomorphism. This situation is denoted by $Y\subseteq X$.
\end{definition}

The internal language of $\C$ can be subsumed in terms of a \emph{Lawvere doctrine}. To avoid considering large sets, we will assume that each $X\in \cC$ has only a small set of subobjects, up to isomorphism. The posets of subobjects $\Sub(X)$ are then meet semilattices where binary meets are computed by pullbacks. Any arrow $f\colon X\to Y$ induces a morphism of meet semilattices $f^{-1}\colon\Sub(Y)\to\Sub(X)$, which pulls back along $f$. This yields an ordinary functor $\Sub \colon \C^\op \to \MSLat$, which we call the \emph{doctrine of subobjects} of $\C$.

\begin{remark}\label{rmk:weighted-lim-internal}
	Let $\C$ be a category with finite weighted limits. Given a finite graph $I$, a diagram $F \colon I \to \C$ and a weight $W \colon I \to \FinPos$, the weighted limit $\lim^W F$ can be defined in the internal language as
	\[ \setst[\Big]{ \overline{x} = (x_{i,w})_{i,w} \in \prod_{i\in I} \prod_{w \in W(i)} F(i) }{ \phi(\overline{x}) \land \psi(\overline{x})} , \]
	where
	\begin{itemize}
		\item $\phi(\overline{x}) = $ ``$x_{i,w} \leq x_{i,w'}$ for all $i \in I$ and all $w \leq w'$ in $W(i)$'';
		\item $\psi(\overline{x}) = $ ``$F(f)(x_{i,w}) = x_{i,W(f)(w)}$ for all $f \colon i \to j$ in $I$ and all $w \in W(i)$''.
	\end{itemize}
	Since only conjunctions and inequalities are used, this definition can be externalised straightforwardly using generalised elements, as explained in Remark~\ref{rmk:external-to-internal-regular}. We then recover the usual definition of the weighted limit $\lim\nolimits^W F$, as given in Section~\ref{s:preliminaries}.
\end{remark}

\begin{lemma}
	The following statements are equivalent for every category~$\C$:
	\begin{enumerate}[label=(\arabic*)]
		\item $\C$ has finite weighted limits.
		\item $\C$ has finite (conical) limits and epi-diagonals.
	\end{enumerate}
\end{lemma}

\begin{proof}
By Remark~\ref{rmk:weighted-lim-internal}, if $\cC$ has finite limits and epi-diagonals, then it has finite weighted limits. Conversely, finite limits and epi-diagonals are special cases of finite weighted limits. Note that the epi-diagonal of $X$ can be obtained as the lax kernel of $\id\colon X\to X$. 
\end{proof}

In the poset-enriched setting, the analogue of a left exact, or finite-limit preserving, functor is a functor that preserves all finite weighted limits. Equivalently, it is a functor that preserves finite limits and epi-diagonals.

\subsection{Regular and coherent categories}

The next step is to add existential quantification to the internal logic.
Let $\cC$ be a category with all finite weighted limits. 

\begin{definition}
The \emph{image} of a morphism $f \colon A \to B$ in $\cC$, if it exists, is the smallest subobject $\im(f)$ of $B$ through which $f$ factors; $f$ is a \emph{surjection} if $\im(f) = B$.
\end{definition} 

We will use the notation $A \epi B$ for surjections. Note that if $f \colon A \to B$ factors as a surjection $A \epi S$ followed by an embedding $S \into B$, then $S$ is the image of $f$. 

\begin{lemma}\label{l:surj-emb-iso}
Isomorphisms coincide with the surjective embeddings.
\end{lemma}

\begin{proof}
	For the non-trivial direction, let $f \colon A \into B$ be a surjective embedding. Then $\im(f) = A$ since $f$ is an embedding, and $\im(f) = B$ since $f$ is surjective.
\end{proof}

As in the ordinary setting, surjections are epimorphisms, and they are left-orthogonal to embeddings in the poset-enriched sense:

\begin{lemma}\label{lem:surj-is-so}
	Let $f \colon A \epi B$ be a surjection. Then $\cA(f,-) \colon \cA(B,-) \to \cA(A,-)$ is an order-embedding. Moreover, \eqref{diag:pullback-orthogonal-surj-emb} is a pullback in $\Pos$ for any embedding $i \colon X \into Y$.
	\begin{equation}\label{diag:pullback-orthogonal-surj-emb}
	\begin{tikzcd}[column sep=3em]
		\cA(B,X) \arrow{r}{\cA(f,X)} \arrow{d}[swap]{\cA(B,i)} & \cA(A,X) \arrow{d}{\cA(A,i)} \\
		\cA(B,Y) \arrow{r}{\cA(f,Y)} & \cA(A,Y)
	\end{tikzcd}\end{equation}
\end{lemma}

\begin{proof}
	Let $u,v \colon B \rightrightarrows C$ be two parallel morphisms. Suppose that $fu \leq fv$. Then $f$ factors through $[u \leq v] \coloneqq \setst{b \in B}{u(b) \leq v(b)} \subseteq B$. Since $f$ is a surjection, $[u \leq v] = B$, which means that $u \leq v$. This proves the first part of the statement.
	
	We now show that \eqref{diag:pullback-orthogonal-surj-emb} is a pullback. Consider the following commutative square.
	\begin{equation}\label{diag:orth-surj-emb}\begin{tikzcd}
			A \ar[r,"f",->>] \ar[d,"u"'] & B \ar[d,"v"]\\
			X \ar[r,hook,"i"] & Y
	\end{tikzcd}\end{equation}
	The pair $(f\colon A \to B, u\colon A \to X)$ factors through the pullback $B \times_Y X$. The arrow $B\times_Y X \to B$ is an embedding because so is $i$, and it is a surjection because so is $f$. Thus, it is an isomorphism by Lemma~\ref{l:surj-emb-iso}, and we get some $w \colon B \to X$ with $fw=u$. Such a $w$ is unique because $f$ is an epimorphism, and $wi = v$ since $fwi = ui = fv$.
	
Assume we have another commutative square as in~\eqref{diag:orth-surj-emb}, with $(u',v')$ instead of $(u,v)$, and suppose that $u \leq u'$ and $v \leq v'$. Then the diagonal filler $w' \colon B \to X$ obtained from $(u',v')$ satisfies $w \leq w'$ by the first part of the statement, because $fw = u \leq u' = fw'$.
\end{proof}

It follows from the previous lemma that surjections coincide with the $2$-categorical notion of \emph{$\so$-morphisms}, where $\so$ stands for \emph{surjective on objects}.

\begin{definition}
	A category with all finite weighted limits is \emph{regular} if every morphism factors as a surjection followed by an embedding, and surjections are stable under pullbacks. A functor is \emph{regular} if it preserves finite weighted limits and surjections.
\end{definition}

\begin{example}
$\Pos$ is a regular category, in which embeddings are order-embeddings, images are the set-theoretic ones, and surjections are surjective order-preserving maps. More generally, if $\cP$ is a small category, then the category of functors $[\cP,\Pos]$ is regular. Weighted limits and images are computed coordinate-wise in $\Pos$, and a morphism is surjective if, and only if, it is component-wise surjective.
\end{example}

Given an arrow $f\colon X\to Y$ and a subobject $\phi \subseteq X$, the image of the composite $\phi \into X \to Y$ is written $f[\phi]$. In terms of the doctrine of subobjects, this defines a left adjoint $f[-] \colon \Sub(X) \to \Sub(Y)$ to $f^{-1}(-) \colon \Sub(Y) \to \Sub(X)$. As in the ordinary case, poset-enriched regularity can also be defined by requiring that each change-of-base morphism $f^{-1} \colon \Sub(Y) \to \Sub(X)$ has a left adjoint $f[-]$, such that the \emph{Beck--Chevalley condition} is satisfied, i.e., for any pullback square as in~\eqref{diag:BC-1}, the square~\eqref{diag:BC-2} commutes.

\begin{minipage}{0.4\textwidth}
	\begin{equation}\label{diag:BC-1}\begin{tikzcd}
			W \ar[r,"u"] \ar[d,"v"'] \arrow[dr, phantom, "\lrcorner", very near start] & Z \ar[d,"g"] \\
			X \ar[r,"f"'] & Y
	\end{tikzcd}\end{equation}
\end{minipage}\hfill
\begin{minipage}{0.55\textwidth}
	\begin{equation}\label{diag:BC-2}\begin{tikzcd}
			\Sub(W) \ar[r,"{u[-]}"] & \Sub(Z) \\
			\Sub(X) \ar[r,"{f[-]}"'] \ar[u,"v^{-1}"] & \Sub(Y) \ar[u,"g^{-1}"']
	\end{tikzcd}\end{equation}
\end{minipage}

The Frobenius rule $f[\phi \land f^{-1} \psi] = f[\phi] \land \psi$ is obtained as a special case. Hence, $\Sub \colon \cC^\op \to \MSLat$ is an ordinary Lawvere doctrine. This allows us to enrich the internal language with existential quantification as follows. If $\phi(x,\ovl{y})$ is a formula in context $x \tcolon X, (y_i \tcolon Y_i)_i$, then $\exists x \qcolon \phi(x,\ovl{y})$ is a formula in context $(y_i \tcolon Y_i)_i$ and it is interpreted as $\pi[\llbracket \phi(x,\ovl{y}) \rrbracket]$, where $\pi \colon X \times \prod_i Y_i \to \prod_i Y_i$ is the canonical projection. 
As we are working with an ordinary Lawvere doctrine, its usual properties apply. For instance, 
\[ 
f[\phi] = \llbracket \exists x \in X \qcolon f(x)=y \land \phi(x) \rrbracket . 
\]
In particular, in the internal language, $f$ is a surjection precisely when ${\proves \exists x \in X\qcolon f(x) = y}$.

In ordinary regular categories, morphisms correspond to functional relations. The same holds for poset-enriched regular categories, as we will now explain.
A relation $R \subseteq X \times Y$ in a category $\cC$ is \emph{functional} if $\proves \exists y \qcolon R(x,y)$ and $(x = x') \land R(x,y) \land R(x',y') \proves y = y'$. We say that $R$ is \emph{order-preserving} if, in addition,\footnote{Note that any relation $R \subseteq X \times Y$ satisfying $\proves \exists y \qcolon R(x,y)$ and eq.~\eqref{eq:order-pres-relation} is automatically functional.}
\begin{equation}\label{eq:order-pres-relation}
(x \leq x') \land R(x,y) \land R(x',y') \proves y \leq y'. 
\end{equation}
The \emph{graph} of a morphism $f \colon X \to Y$ is $\setst{(x,y) \in X\times Y}{y = f(x)}$. As with arbitrary regular Lawvere doctrines, we obtain a functor from the base category $\cC$ to its category of functional relations, with the same objects as $\cC$ and morphisms the functional relations. 

\begin{proposition}\label{p:functional-rels}
	If $\cC$ is a regular category, then it is equivalent to its category of order-preserving functional relations.
\end{proposition}

\begin{proof}
	Let $f \colon X \to Y$ be a morphism in $\cC$. The graph of $f$ is order-preserving because $x \leq x' \proves f(x) \leq f(x')$. Now, let $R \subseteq X\times Y$ be an order-preserving functional relation. The condition $\proves \exists y \qcolon R(x,y)$ means that the composite $R\into X\times Y \to X$ is a surjection, while the condition in eq.~\eqref{eq:order-pres-relation} implies that it is an embedding. Just observe that, if $(x,y), (x',y') \in R$ and $x \leq x'$, then $y \leq y'$ and so $(x,y) \leq (x',y')$. Hence, $R \to X$ is an isomorphism by Lemma~\ref{l:surj-emb-iso}, and we obtain a morphism $X \to Y$ whose graph is $R$. 
\end{proof}

We will now consider coherent categories, thereby adding finite disjunctions to the internal logic.

\begin{definition}
A regular category $\cC$ is \emph{coherent} if, for every $X\in \cC$, the meet semilattice $\Sub(X)$ has finite suprema and these are preserved under pullbacks. A~functor is \emph{coherent} if it preserves finite weighted limits, images, and finite suprema of subobjects.
\end{definition}

As in the ordinary case, any coherent category $\cC$ has a strict initial object, characterised by $\top \proves \bot$ (i.e., there is only one subobject). We call this the \emph{empty object}. The reasoning is the same as in the ordinary case: let $I$ be the least subobject of any object. We claim that any arrow $f \colon X \to I$ is an isomorphism. Firstly, since $f^{-1}$ preserves both $\top$ and $\bot$, $X$ also satisfies $\top \proves \bot$. Similar reasoning applies to $X\times X$, so $f(x) \leq f(x') \proves x \leq x'$ since $X \times X$ has only one subobject. Thus, $f$ is an embedding; but it is also surjective, thus an isomorphism, because $I$ has only one subobject. Finally, note that for any object $A$, there is a morphism $I \to A$, since the projection $A \times I \to I$ is an isomorphism. Moreover, any two morphisms $I \to A$ are equal, since their equaliser must be $\top$. 

Similarly to the ordinary case, it is not difficult to see that, since suprema of subobjects are stable under pullbacks, the doctrine of subobjects $\Sub \colon \C^\op \to \MSLat$ takes values in the category $\DL$ of (bounded) distributive lattices and lattice homomorphisms.

\subsection{Pretoposes: quotients and disjoint unions}
In analogy with the ordinary case, we introduce pretoposes as ``definable completions'' of coherent categories, adding quotients by congruences,\footnote{Not every congruence, as defined below, is a congruence in the sense of~\cite{Street1982}. Instead, our definition is equivalent to those adopted in, e.g., \cite{KV2017,AV2022,Adamek2025}.} and disjoint unions.

\begin{definition}
	Let $\cC$ be a regular category. A \emph{congruence} on an object $X$ is a relation $R \subseteq X^2$ that is transitive and satisfies $x\leq y \proves R(x,y)$. If $R$ is a congruence on $X$, a \emph{quotient} of $X$ by  $R$ is a surjection $X \epi X/R$ whose lax kernel is $R$.
\end{definition}

Every congruence is reflexive, because $\proves x\leq x$. Further, the lax kernel of any arrow is a congruence, and any surjection is a quotient of its lax kernel. 
The following is a rephrasing of \cite[Proposition~2.16]{AV2022}, stating that in a regular category, every quotient is the coinserter of its lax kernel. We offer a proof using the internal language.

\begin{lemma}\label{l:quotient-coinserter}
If it exists, a quotient $X \epi X/R$ is a coinserter of $R \rightrightarrows X$.
\end{lemma}

\begin{proof}
	Let $q \colon X \epi X/R$ be the quotient arrow. We know that $uq \leq vq$ since $(u,v)$ is the lax kernel of $q$. Since $q$ is a surjection, it is an epimorphism by Lemma~\ref{lem:surj-is-so}. It remains to show that any arrow $f \colon X \to Z$ such that $uf \leq vf$ factors through $q$.
	\[\begin{tikzcd}
		R \ar[r,shift left=3pt,"u"] \ar[r,shift right=3pt,"v"'] & X \ar[r,->>,"q"] \ar[rd,"f"'] & X/R \ar[d,dashed] \\
		& & Z
	\end{tikzcd}\]
We define an order-preserving functional relation $G \subseteq X/R \times Z$ by
	\[ G(y,z) \iff \exists x \in X \qcolon q(x) = y \land f(x) = z . \]
Since $q$ is surjective, we have $\proves \exists z \qcolon G(y,z)$. As $(u,v)$ is the lax kernel of~$q$, the fact that $uf \leq vf$ is equivalent to $q(x) \leq q(x') \proves f(x) \leq f(x')$. Thus, $G$ is order-preserving:
\[\begin{aligned}
		&&& y \leq y' \land [q(x)=y \land f(x) = z] \land [q(x')=y' \land f(x')=z']\\
		\proves &&& q(x) \leq q(x') \land f(x) = z \land f(x') = z'\\
		\proves &&& z \leq z' .
	\end{aligned}\]
By Proposition~\ref{p:functional-rels}, there is a morphism $g \colon X/R\to Z$ whose graph is $G$. The definition of $G$ implies that $\proves g(q(x)) = f(x)$.
\end{proof}

\begin{definition}
	A category is \emph{exact} if it is regular and every congruence has a quotient.
\end{definition}

\begin{example}\label{ex:exact-cats}
For any small category $\cP$, the presheaf category $[\cP,\Pos]$ is exact. The category~$\KOrd$ of compact ordered spaces is also exact. On the other hand, $\Set$ (regarded as a poset-enriched category with discrete hom-sets) is regular but not exact, because congruences coincide with preorders. See e.g.\ \cite[\S 3.2]{KV2017} and \cite[Proposition~5.8]{AV2022}.
\end{example}

Next, we will consider disjoint unions.

\begin{definition}\label{dfn:disjoint-union}
	In a coherent category $\cC$, the \emph{disjoint union} of objects $A$ and $B$ is an object $A+B$, equipped with two embeddings $A \into A+B$ and $B \into A+B$, such that $A$ and $B$ cover $A+B$ and are incomparable. In the internal language, where for instance ``$x \in A$'' refers to the predicate $A \into A+B$, this can be expressed as follows:
	\begin{itemize}
		\item $\proves (x \in A) \lor (x \in B)$
		\item $(a \in A) \land (b \in B) \land (a \leq b) \proves \bot$
		\item $(a \in A) \land (b \in B) \land (b \leq a) \proves \bot$
	\end{itemize}
\end{definition}

\begin{lemma}
If it exists, the disjoint union of $A$ and $B$ is their coproduct.
\end{lemma}

\begin{proof}
Consider the canonical embeddings $\iota_A \colon A \hookrightarrow A+B \hookleftarrow B \cocolon \iota_B$. We claim that the latter is a coproduct diagram, i.e., for every $X\in \cC$ the induced map
\[
\cC(A+B,X) \to \cC(A,X) \times \cC(B,X)
\] 
is an order-isomorphism. Since $A$ and $B$ jointly cover $A+B$, the same reasoning that was used to prove the first part of Lemma~\ref{lem:surj-is-so} shows that the above map is an order-embedding. It remains to show that it is also surjective. Consider arbitrary arrows $f \colon A \to X$ and $g \colon B \to X$. We define a relation $R \subseteq (A+B)\times X$ by
	\[ R(s,x) \iff (\exists a \in A \qcolon s=\iota_A(a) \land f(a)=x) \lor (\exists b \in B \qcolon s=\iota_B(b) \land f(b)=x) . \]
We have $\proves \exists x \qcolon R(s,x)$ because, for every $s \in A+B$, either $s \in \im(\iota_A)$ or $s \in \im(\iota_B)$. To show that $R(s,x)$ is order-preserving, suppose that $s \leq s'$, $R(s,x)$ and $R(s',x')$. We must prove that $x \leq x'$. Depending on whether each of $s$ and $s'$ is in $\im(\iota_A)$ or $\im(\iota_B)$, there are four cases to consider. If $s = \iota_A(a)$ and $s' = \iota_A(a')$, then $x = f(a) \leq f(a') = x'$ since $\iota_A$ is an embedding. If $s = \iota_A(a)$ and $s' = \iota_B(b')$, then $s \leq s'$ is impossible. The other two cases are symmetric. By Proposition~\ref{p:functional-rels}, $R$ is the graph of an arrow $[f,g] \colon A+B\to X$. We have $f(a) = x \proves [f,g](\iota_A(a)) = x$, and similarly for $g$ and $\iota_B$. It follows that $\iota_A [f,g] = f$ and $\iota_B [f,g] = g$, showing that the above map is surjective.
\end{proof}

\begin{definition}
	A category is a \emph{pretopos} if it is coherent, exact, and has disjoint~unions.
\end{definition}

\begin{example}
All of the exact categories given in Example~\ref{ex:exact-cats}, namely the presheaf categories $[\cP,\Pos]$ and $\KOrd$, are also pretoposes.
\end{example}

\begin{remark}
In ordinary pretoposes, pushouts cannot be characterised using the internal language because they require an infinite iteration. By contrast, quotients and disjoint unions in poset-enriched pretoposes can be combined to form lax pushouts.

In the internal logic, the lax pushout $(A+B)/R$ of arrows $f \colon R\to A$ and $g \colon R\to B$ is characterised by two embeddings $\iota_A \colon A \into (A+B)/R$ and $\iota_B \colon B \into (A+B)/R$ subject to the conditions: (i) $\proves (x \in A) \lor (x \in B)$, (ii) ${\iota_A(a) \leq \iota_B(b) \proves \exists r \in R \qcolon a=f(r) \land b=g(r)}$, and conversely (iii) $\exists r \in R \qcolon a=f(r) \land b=g(r) \proves \iota_A(a) \leq \iota_B(b)$.
\end{remark}

\subsection{Constants}\label{s:constants}
Let $\one$ denote the terminal object of a regular category $\cC$. Given $X \in \cC$, a \emph{constant}, or \emph{point}, of $X$ is an arrow $c \colon \one \to X$. The internal logic of $\cC$ can be enriched with these constants in the obvious way, thanks to the canonical identification $\one \times A = A$. Any constant $c \colon \one \to X$ is an embedding because $\cC(T,\one)$ is a singleton for all~$T$, and so $\cC(T,\one) \to \cC(T,X)$ is an order-embedding. We write $\set{c}$ for the image of $c \colon \one \to X$. 

\begin{example}\label{exmp:two-tensor}
The lax pushout of the identity $\one \to \one$ along itself is characterised by two constants $0$ and $1$ subject to the conditions $\proves (x=0) \lor (x=1)$, $\proves 0 \leq 1$ and $1 \leq 0 \proves \bot$.
This object is also the tensor $\two \cdot \one$, where $\two$ is the two-element chain, characterised by the universal property $\cC(\two\cdot\one,X) \cong \Pos(\two,\cC(\one,X))$ for every $X\in\cC$. Tensors with finite posets can be built in a similar way.
\end{example}

\section{Projective covers and generators}\label{s:projectives-and-generators}

With the aim of characterising the category of compact ordered spaces in Section~\ref{s:char-KOrd}, in the present section, we study projective objects and generators in regular categories.
In particular, we show that any exact category with enough projectives can be reconstructed, up to equivalence, from its full subcategory of projective objects---and, in fact, from any of its projective covers (Theorem~\ref{t:presh-emb}). This can be regarded as a weak version of Barr's embedding theorem for ordinary regular categories.

\subsection{A weak version of Barr's embedding theorem}

\begin{definition}
	Let $\A$ be a regular category. An object $X \in \A$ is \emph{projective} if $\A(X,-) \colon \A \to \Pos$ preserves surjections. A full subcategory $\Ps \subseteq \A$ is a \emph{projective cover} of $\A$ if every object in $\Ps$ is projective and each object of $\A$ is \emph{covered} by an object in $\Ps$, i.e.\ for every $X\in \A$ there exist $Y\in \Ps$ and a surjection $Y\epi X$.
\end{definition}

Vitale showed in \cite[Proposition~2.3]{Vitale1994} that two ordinary exact categories with enough projectives are equivalent if, and only if, they admit equivalent projective covers. Corollary~\ref{c:proj-cover-equi} establishes an analogous result for poset-enriched exact categories. This relies on Theorem~\ref{t:presh-emb}, which is a variant, for regular categories with enough projectives, of Barr's embedding theorem for ordinary regular categories~\cite{Barr1971}; see Remark~\ref{rem:Barr-embedding-thm}.

\begin{theorem}\label{t:presh-emb}
	Let $\cA$ be a regular category, and let $\Ps \subseteq \cA$ be a projective cover. Then the nerve functor 
\[N_\Ps \colon \cA \to [\Ps^\op,\Pos]\]
is regular and fully faithful. Moreover, if $\cA$ is exact, the essential image of $N_\Ps$ consists of those presheaves that are obtained as the quotient of a representable presheaf by a congruence that is covered by another representable presheaf.
\end{theorem}

\begin{proof}
	The nerve functor $N_\Ps \colon \cA \to [\Ps^\op,\Pos]$ preserves finite weighted limits for any full subcategory $\Ps \subseteq \cA$, and it preserves surjections if and only if each object of $\Ps$ is projective (just recall that the surjections in $[\Ps^\op,\Pos]$ are the component-wise surjections). Hence, $N_{\Ps}$ is a regular functor.
Now, let $f, g \colon X\rightrightarrows Y$ in $\cA$ satisfy $N_\Ps(f) \leq N_\Ps(g)$, and let $q \colon P \epi X$ be a surjection with $P \in \Ps$. Since $qf \leq qg$ and $q$ is a surjection, we get $f\leq g$ by Lemma~\ref{lem:surj-is-so}. This shows that $N_\Ps$ is faithful.
	
To prove that $N_\Ps$ is full, let $F \colon N_\Ps(X) \to N_\Ps(Y)$ be a natural transformation. In the spirit of the Yoneda lemma, consider $F(q) \in \cA(P,Y)$. We will show that there exists a morphism $f \colon X \to Y$ making the following diagram commute.
	\[\begin{tikzcd}
		P \ar[d,->>,"q"'] \ar[rd,"F(q)"] & \\
		X \ar[r,dashed,"f"'] & Y
	\end{tikzcd}\]
	Since $\Ps$ is a projective cover, there exists $Q \in \Ps$ that covers the lax kernel $\lker(q)$. 
		Let $a, b \colon Q \rightrightarrows P$ be the two induced morphisms, which satisfy $aq \leq bq$. By naturality of $F$, we get $aF(q) = F(aq) \leq F(bq) = bF(q)$. This shows that $\lker(q) \subseteq \lker(F(q))$. Hence, there is $f \colon X \to Y$ as desired. We claim that $F = N_\Ps(f)$. Let $g \colon P' \to X$ with $P' \in \Ps$. By projectivity, there is $g' \colon P' \to P$ such that $g'q = g$. Thus, $F(g) = F(g'q) = g'F(q) = g'qf = gf$.
	This concludes the proof that $N_\Ps$ is regular and fully faithful.
	
Now, assume $\cA$ is exact. As we have just seen, each object $X \in \cA$ has a ``presentation'' 
\[\begin{tikzcd}
Q \arrow[yshift=4pt]{r}{a} \arrow[yshift=-4pt]{r}[swap]{b} & P \arrow[twoheadrightarrow]{r}{q} & X
\end{tikzcd}\] 
with $P, Q \in \Ps$. The image of $Q \rightrightarrows P$ (more precisely, of the induced arrow ${Q\to P^{2}}$) is the congruence $\lker(q)$, therefore~$X$ is the quotient of $P$ by this image. Since regular functors preserve congruences, $N_\Ps(X)$ is the quotient of the image of $N_\Ps(Q) \rightrightarrows N_\Ps(P)$. 
Conversely, given arrows $N_\Ps(Q) \rightrightarrows N_\Ps(P)$ whose image is a congruence, the image of $Q \rightrightarrows P$ is also a congruence, because $N_\Ps$ is fully faithful and thus conservative. Since $\cA$ is exact, the quotient of $Q \rightrightarrows P$ exists and is sent to the quotient of $N_\Ps(Q) \rightrightarrows N_\Ps(P)$, which is therefore in the image of $N_\Ps$. This proves the second part of the theorem.
\end{proof}

\begin{remark}
The proof that $N_\Ps \colon \cA \to [\Ps^\op,\Pos]$ is a fully faithful regular functor can be phrased using the $(1,2)$-categorical notion of a Grothendieck topology. For further information on $2$-categorical sheaf theory, we refer the interested reader to~\cite{Street1982}.

Let $J_{\mathrm{reg}}$ be the topology on $\cA$ generated by the surjections. We have the following chain of embeddings and equivalences:
	\[ \cA \subseteq \Sh(\cA,J_{\mathrm{reg}}) \simeq \Sh(\Ps,J_{\mathrm{reg}}) = [\Ps^\op,\Pos] . \]
	The first embedding $\cA \subseteq \Sh(\cA,J_{\mathrm{reg}})$ is fully faithful and regular. Since every object of $\cA$ is covered by an object of $\Ps$, we have $\Sh(\cA,J_{\mathrm{reg}}) \simeq \Sh(\Ps,J_{\mathrm{reg}})$, where the second topology is the restriction of the first one. Finally, in $(\Ps,J_{\mathrm{reg}})$ every cover is a split epimorphism, hence the topology is trivial and $\Sh(\Ps,J_{\mathrm{reg}}) = [\Ps^\op,\Pos]$.
\end{remark}

\begin{remark}\label{rem:Barr-embedding-thm}
	Theorem~\ref{t:presh-emb} can be seen as a variant of Barr's embedding theorem that applies to (poset-enriched) regular categories with enough projectives. Barr's embedding theorem states that if $\cA$ is an ordinary regular category, then the regular functor $\cA \to [\cMod(\cA)^\op,\Set]$ is fully faithful, where $\cMod(\cA)$ is the category of \emph{models} of $\cA$, i.e., the regular functors $\cA \to \Set$. An object $P$ of $\cA$ is projective exactly when $\cA(P,-) \colon \cA \to \Set$ is a regular functor. Hence, the projective objects can be considered a particularly simple type of model, sitting at the intersection of syntax and semantics. When there are enough projectives, we can replace the category of models $\cMod(\cA)$ in Barr's embedding theorem with its full subcategory determined by the projective objects.
\end{remark}

Theorem~\ref{t:presh-emb} has the following immediate consequence.

\begin{corollary}\label{c:proj-cover-equi}
Let $\cA, \cB$ be exact categories, and let $\Ps_{\cA} \subseteq \cA$ and $\Ps_{\cB} \subseteq \cB$ be projective covers. If $\Ps_{\cA} \simeq \Ps_{\cB}$ then $\cA\simeq \cB$.
\end{corollary}

In $\Pos$, every object is covered by a set, and every set is projective. Thus, $\Set\into \Pos$ is a projective cover (in fact, every projective object in $\Pos$ is order-discrete). The following is a generalisation of the observation that sets are projective in~$\Pos$.

\begin{lemma}\label{projective-tensors}
Let $X$ be a projective object in a regular category $\cA$. For any set $S$, if the copower $S\cdot X$ exists, then it is projective.
\end{lemma}
\begin{proof}
The functors $\cA(X,-)$ and $\Pos(S,-)$ preserve surjections because $X$ and $S$ are projective in $\cA$ and $\Pos$, respectively. Therefore, their composite $\Pos(S,\cA(X,-))\cong \cA(S\cdot X, - )$ preserves surjections. That is, $S\cdot X$ is projective.
\end{proof}

\subsection{Generators and discrete generators}
We start by recalling the customary definition of a generator, cf.\ e.g.~\cite{KV2017}.

\begin{definition}\label{d:generator}
An object $G$ of a regular category $\cA$ is a \emph{generator} if it satisfies the following properties:
\begin{enumerate}[label=(\roman*)]
\item for every poset $P$, the tensor $P\cdot G$ exists in $\A$;
\item\label{i:gen-covers} for every $X\in\cA$, the canonical arrow $\A(G,X)\cdot G \to X$ is a surjection.
\end{enumerate}
\end{definition}

Given a \emph{projective generator} $G$ in $\cA$, i.e., a generator that is also a projective object, we want to construct a projective cover of $\cA$. To this end, consider the full subcategory of $\cA$ spanned by the objects of the form $P\cdot G$, with $P$ a poset. Since $G$ is a generator, every object of $\cA$ is covered by an object of this form. But in general, objects of the form $P\cdot G$ need not be projective, unless $P$ is order-discrete (Lemma~\ref{projective-tensors}). This naturally leads us to consider copowers of $G$, and the following variant of the notion of generator.

\begin{definition}\label{d:discrete-generator}
An object $G$ of a regular category $\cA$ is a \emph{discrete generator} if it satisfies the following properties:
\begin{enumerate}[label=(\roman*)]
\item for every set $S$, the copower $S\cdot G$ exists in $\A$;
\item\label{i:discrete-gen-covers} for every $X\in\cA$, the canonical arrow $\abs{\A(G,X)}\cdot G \to X$ is a surjection.
\end{enumerate}
\end{definition}

\begin{remark}
Item~\ref{i:gen-covers} in Definition~\ref{d:generator} is equivalent to saying that for every $X\in \cA$ there exist a poset $P$ and a surjection $P\cdot G\epi X$. Similarly, item~\ref{i:discrete-gen-covers} in Definition~\ref{d:discrete-generator} is equivalent to saying that for every $X\in \cA$ there exist a set $S$ and a surjection $S\cdot G\epi X$.
\end{remark}

The discussion above yields the following fact.

\begin{lemma}\label{l:projective-discrete-gen-cover}
Let $\cA$ be a regular category and $G$ a projective discrete generator in~$\cA$. The full subcategory of $\cA$ spanned by the copowers $S\cdot G$, with $S$ a set, is a projective~cover. 
\end{lemma}

The next proposition shows that, for regular categories, the concept of a discrete generator is a weakening of that of a generator. However, the two notions coincide for exact categories.

\begin{proposition}\label{p:gen-iff-dicrete-gen-and-coinserters}
The following statements hold in any regular category $\cA$:
\begin{enumerate}[label=(\alph*)]
\item\label{i:gen-implies-dis-gen} Every generator is a discrete generator.
\item\label{i:dis-gen-implies-gen} If $\cA$ is exact, then every discrete generator is a generator.
\end{enumerate}
Moreover, if $\cA$ is exact and admits a generator, then it is well-powered and cocomplete.
\end{proposition}

\begin{proof}
\ref{i:gen-implies-dis-gen} Suppose $G$ is a generator. Since all tensors $P\cdot G$ exist, there is an adjunction
	\[\begin{tikzcd}
		\Pos \arrow[yshift=-5pt]{rr}[swap]{-\cdot G} & \text{\tiny{$\top$}} & \A \arrow[yshift=5pt]{ll}[swap]{\A(G,-)}.
	\end{tikzcd}\]
The right adjoint preserves (finite) weighted limits, and in particular embeddings, and so the left adjoint $-\cdot G$ preserves surjections (see, e.g., \cite[Proposition~2.1.12]{Cisinski2019}). Therefore, the identity map $\abs{\A(G,X)}\epi \A(G,X)$ in $\Pos$ induces a surjection 
\[
\abs{\A(G,X)}\cdot G \epi \A(G,X)\cdot G.
\] 
If $X$ is any object of $\cA$, the canonical map $\A(G,X)\cdot G\to X$ is a surjection, and so the composite yields a surjection $\abs{\A(G,X)}\cdot G\epi X$. Hence, $G$ is a discrete generator.

\ref{i:dis-gen-implies-gen} Assume that $\cA$ is exact, and let $G$ be a discrete generator. To show that $G$ is a generator, it is enough to prove that the tensor $P\cdot G$ exists for every poset $P$. This, in turn, follows if we prove the second part of the statement, namely that $\cA$ is well-powered and cocomplete.
Fix an arbitrary $X \in \cA$, and write $\P(T)$ for the power-set of a set~$T$. We claim that there is a Galois connection
	\[\begin{tikzcd}
		\Sub(X) \ar[r,shift left=2.5pt,"R"] & \P(|\cA(G,X)|) , \ar[l,shift left=2.5pt,"L"]
	\end{tikzcd}\]
where $R$ sends $Y \subseteq X$ to the set of arrows $G\to X$ that factor through $Y$, and $L$ sends $S \subseteq |\cA(G,X)|$ to the image of $S\cdot G \to X$. Since every $Y \subseteq X$ is covered by $|\cA(G,Y)| \cdot G$, we have $L(R(Y)) = Y$. Furthermore, every morphism in $S \subseteq |\cA(G,X)|$ factors through the image of $S \cdot G \to X$, that is $S \subseteq R(L(S))$. Hence, $\Sub(X)$ is a reflective sub-poset of $\P(|\cA(G,X)|)$. This entails that $\Sub(X)$ is small and complete.
	
Next, we prove that $\cA$ admits coinserters and small coproducts, hence it is cocomplete; cf.\ \cite[Proposition~5.2]{Kelly1989}. 
Consider arrows $u,v \colon Y \rightrightarrows X$ in $\cA$. For any relation $R\subseteq X^{2}$, define the functor
\[ 
C_{R}\colon \cA\to \Pos, \ \ C_R(Y) \coloneqq \setst{f\colon X \to Y}{R \subseteq \lker(f)} . 
\]
 The universal property of the coinserter of $u$ and $v$ is to represent the functor $C_{\im(u,v)}$, where the relation $\im(u,v)$ is the image of $u,v \colon Y \rightrightarrows X$. We claim that $C_R$ is always representable. In particular, it will follow that the coinserter of $u$ and $v$ exists.

If $R$ is a congruence, then $C_{R}$ is representable by Lemma~\ref{l:quotient-coinserter}, because $\cA$ is exact. In general, note that $R$ is a congruence exactly when it is a fixpoint of
	\[ \Sub(X^2)\to \Sub(X^2), \ \ R \mapsto [\leq_X] \lor R \lor (R\circ R) . \]
This map is expansive and order-preserving. As the poset $\Sub(X^2)$ is complete, for every $R \subseteq X^2$ there is a smallest fixpoint $R^*$ above $R$, by the Knaster--Tarski fixpoint theorem. Moreover, since $\lker(f)$ is a congruence, $C_R = C_{R^*}$. Thus, since $C_{R^*}$ is representable, so is $C_R$. We shall call the object representing $C_R$ the \emph{quotient} of $X$ by~$R$.
	
	It remains to show that $\cA$ admits small coproducts. Let $(A_i)_i$ be a set of objects of $\cA$. For each $i$, let $T_i\cdot G \rightrightarrows S_i\cdot G \epi A_i$ be a presentation of $A_i$ as in the proof of Theorem~\ref{t:presh-emb}, with $S_{i}$ and $T_{i}$ sets. The coproduct of the $(A_i)_i$ is then the quotient of $(\coprod_i S_i)\cdot G$ by the union of the images of the arrows $Y_j\cdot G \rightrightarrows (\coprod_i S_i) \cdot G$.
\end{proof}

Although this fact will not be needed later on, we will prove that any exact category~$\cA$ that has a projective generator~$G$ is monadic over $\Pos$. When $\cA$ has coinserters and~$G$ is presentable, this follows from \cite[Theorem~5.9]{KV2017}. 
We offer a proof that does not rely on the aforementioned theorem (which uses a result by Bloom and Wright~\cite{BW83}), but rather on Theorem~\ref{t:presh-emb}.
The analogue of this result for ordinary categories was proved in \cite[Proposition~3.2]{Vitale1994}, as part of a characterisation of categories monadic over $\Set$.

\begin{theorem}\label{t:KV-charact}
If $\A$ is an exact category admitting a projective (discrete) generator~$G$, then the following adjunction is monadic:
	\[\begin{tikzcd}
		\Pos \arrow[yshift=-5pt]{rr}[swap]{-\cdot G} & \text{\tiny{$\top$}} & \A \arrow[yshift=5pt]{ll}[swap]{\A(G,-)}.
	\end{tikzcd}\]
\end{theorem}

\begin{proof}
Write $\T$ for the monad on $\Pos$ induced by the above adjunction, and $\EM(\T)$ for its category of Eilenberg--Moore algebras. By Lemma~\ref{l:projective-discrete-gen-cover}, the full subcategory $\cB \subseteq \cA$ spanned by the copowers $S\cdot G$, with $S$ a set, is a projective cover.

Similarly, let $\cK\subseteq\EM(\T)$ be the full subcategory spanned by the copowers of $F$, the free algebra on a one-element poset; note that $\cK$ can be identified with the full subcategory consisting of the algebras free on discrete posets. The monad $\T$ preserves surjections because~$G$ is projective and $-\cdot G$ preserves surjections (for the latter fact, cf.~the proof of Proposition~\ref{p:gen-iff-dicrete-gen-and-coinserters}\ref{i:gen-implies-dis-gen}). It follows from (the proof of) \cite[Corollary~4.14]{KV2017} that $\EM(\T)$ is regular and the forgetful functor $U\colon \EM(\T)\to \Pos$ reflects weighted limits. Further, $U$ preserves and reflects surjections, and its left adjoint preserves them; cf.~Proposition~4.12 in \emph{op.~cit}. This implies that $F$ is a projective discrete generator, and so $\cK$ is a projective cover of $\EM(\T)$ by Lemma~\ref{l:projective-discrete-gen-cover}.

The full inclusion of the Kleisli category of $\T$ into $\A$ restricts to an equivalence $\cK\simeq \cB$, which is also the restriction of the comparison functor $\cA \to \EM(\T)$. The latter is a regular functor. Just observe that it preserves finite weighted limits and surjections because $\A(G,-)\colon \A\to\Pos$ preserves them and $U\colon \EM(\T)\to \Pos$ reflects them. Identifying $\cK$ and $\cB$ as one and the same category $\Ps$, by Theorem~\ref{t:presh-emb} we get a commutative diagram of regular functors as displayed below.
\[\begin{tikzcd}[column sep = 0.3em]
& {[\Ps^{\op},\Pos]} \arrow[hookleftarrow]{dr} & \\
\A \arrow{rr} \arrow[hookrightarrow]{ur} & & \EM(\T) 
\end{tikzcd}\]
In particular, the comparison functor $\cA\to\EM(\T)$ is fully faithful. Moreover, since $\cA$ is exact, it is equivalent to the full subcategory $\cC\subseteq [\Ps^{\op},\Pos]$ consisting of those presheaves that are obtained as the quotient of a representable presheaf by a congruence that is covered by another representable presheaf. The proof of Theorem~\ref{t:presh-emb} shows that even in the absence of exactness, the essential image of $\EM(\T)\into [\Ps^{\op},\Pos]$ is contained in $\cC$.

Since the composite $\cA\into \EM(\T)\into \cC$ is an equivalence, we conclude that the comparison functor $\cA\to \EM(\T)$ is an equivalence.
\end{proof}

\begin{remark}
From Theorem~\ref{t:KV-charact}, we can recover the result of Flagg stating that the underlying-poset functor $\KOrd\to\Pos$ is monadic~\cite[Theorem~12]{Flagg1997}. This implies that $\KOrd$ is equivalent to the category of algebras for the prime filter monad on $\Pos$, which takes a poset $P$ to the poset of prime filters of the lattice of upsets of $P$. 

Just recall from~\cite[Proposition~5.8]{AV2022} that $\KOrd$ is exact. Since its terminal object~$\one$ is a projective (discrete) generator, by Theorem~\ref{t:KV-charact}, the following adjunction is monadic:
	\[\begin{tikzcd}
		\Pos \arrow[yshift=-5pt]{rr}[swap]{-\cdot \one} & \text{\tiny{$\top$}} & \KOrd \arrow[yshift=5pt]{ll}[swap]{\KOrd(\one,-)}.
	\end{tikzcd}\]
Thus, the functor $\KOrd(\one,-)$, which is naturally isomorphic to the underlying-poset functor, is monadic. It is also not difficult to verify directly that, for each poset $P$, the tensor $P\cdot\one$ is isomorphic to the \emph{Stone--\v{C}ech--Nachbin order compactification} of $P$, cf.~\cite{Nachbin1965}, whose elements are the prime filters of the lattice of upsets of $P$.
\end{remark}

\section{Well-pointed cocomplete pretoposes}\label{s:well-pointed-pretoposes}

In this section, we study pretoposes whose terminal object $\one$ is a generator.
By Lemma~\ref{l:one-proj} below, this implies that $\one$ is a projective generator, and so the pretopos is ``controlled'' by its projective cover $\setst{S\cdot\one}{S\in\Set}$ by Corollary~\ref{c:proj-cover-equi}. Further, we shall see that these pretoposes are precisely the well-pointed cocomplete pretoposes.
For any non-degenerate pretopos in this class, we will construct a functor into $\KOrd$ that sends $S\cdot\one$ to the Stone--\v{C}ech compactification $\beta(S)$. While the functor itself is not an equivalence, it will allow us to characterise $\KOrd$ as a pretopos in Section~\ref{s:char-KOrd}.

Throughout this section, we fix an arbitrary pretopos $\cC$. We say that $\cC$ is \emph{degenerate} if $\proves \bot$, i.e., $\cC$ is equivalent to the category with only one object and one morphism. Given $X\in \cC$, recall from Section~\ref{s:constants} that a point, or constant, of $X$ is an arrow $\one \to X$; each such point is an embedding. 

\begin{lemma}\label{l:one-proj}
If $\one$ is a generator of $\cC$, then it is projective.
\end{lemma}

\begin{proof}
	Let $f \colon X \epi Y$ be a surjection and let $p \colon \one \to Y$ be a point. The unique arrow $f^{-1}(p) \to \one$ is surjective because surjections are stable under pullbacks. If $f^{-1}(p)$ has no points, then it is empty, since $\one$ is a generator. In this case, $\one$ is also empty and the pretopos is degenerate, so $\one$ is trivially projective.  Otherwise, $f^{-1}(p)$ has a point $\one \to f^{-1}(p) \subseteq X$, which yields a point $q\colon \one \to X$ satisfying $qf=p$.
\end{proof}

\begin{definition}
$\cC$ is \emph{well-pointed} if for every subobject $\phi \subseteq X$, we have $\proves \phi(x)$ if and only if, for every point $c \colon \one \to X$, $\proves \phi(c)$.
\end{definition}

\begin{proposition}\label{prop:well-pointed-cocomplete}
$\cC$ is well-pointed and cocomplete if, and only if, $\one$ is a generator.
\end{proposition}

\begin{proof}
	Suppose that $\one$ is a generator. By Proposition~\ref{p:gen-iff-dicrete-gen-and-coinserters}, $\cC$ is cocomplete. Moreover, if $X\in \cC$, and $Y$ is a subobject of $X$ that contains all the points $\one\to X$, then it contains the image of $\cC(\one,X)\cdot\one\epi X$. That is, $Y=X$. This shows that $\cC$ is well-pointed.
	
	For the other direction, suppose that $\cC$ is well-pointed and cocomplete. In particular, the tensor $P\cdot\one$ exists for every poset $P$. For each $X\in \cC$, the image of $\cC(\one,X)\cdot\one\to X$ contains all the points of $X$, hence it covers $X$ because $\cC$ is well-pointed.
\end{proof}

\begin{remark}
Direct inspection shows that Lemma~\ref{l:one-proj} and Proposition~\ref{prop:well-pointed-cocomplete} hold not only for pretoposes, but more generally for exact categories.
\end{remark}

To define a functor $\cC\to \KOrd$, whenever $\cC$ is non-degenerate and its terminal object is a generator, we consider upward complemented subobjects in $\cC$. These can be defined in any regular category as follows.

\begin{definition}
	In a regular category, each subobject $U \subseteq X$ has an \emph{upward closure} $\up U \subseteq X$ defined as $\setst{x\in X}{\exists u \in U \qcolon u \leq x}$. If $U = \up U$, we say that $U$ is \emph{upward closed}, or that it is an \emph{upward subobject} of $X$. Moreover, a subobject $U \subseteq X$ is \emph{complemented} if it has a complement in $\Sub(X)$. We denote by $\CU(X)$ the sub-poset of $\Sub(X)$ consisting of the upward complemented subobjects of $X$.
\end{definition}

It is not difficult to see, using the internal logic, that for any object $X\in \cC$, $\CU(X)$ is a sublattice of $\Sub(X)$. Furthermore, if $f \colon X \to Y$, then $f^{-1}$ restricts to a lattice homomorphism $\CU(Y) \to \CU(X)$. Unlike $\Sub \colon \cC^\op \to \DL$, the presheaf 
\[
\CU\colon  \cC^\op \to \DL
\]
is poset-enriched. Just observe that, if $f,g\in\C(X, Y)$ satisfy $f\leq g$, then $f(x) \leq g(x)$ and so $f(x) \in U \proves g(x) \in U$ for any $U \in \CU(Y)$. That is, $f^{-1}\leq g^{-1}$.

For the next lemma, recall from Example~\ref{exmp:two-tensor} that $\two$ is the two-element poset $\set{0 < 1}$, and the tensor $\two\cdot\one$ can be characterised in the internal logic as the only object with two constants $0$ and $1$ such that $\proves (x=0)\lor(x=1)$, $\proves 0\leq 1$ and $1\leq 0\proves\bot$. 
\begin{lemma}\label{lem:clopen-up-rep}
	The presheaf $\CU\colon  \cC^\op \to \DL$ is represented by~$\two\cdot\one$.
\end{lemma}

\begin{proof}
	The universal element, witnessing the representability of $\CU$, is $\set{1} \in \CU(\two\cdot\one)$. It is indeed an upward complemented subobject: its complement is $\set{0}$, and if $1 \leq x$ then either $x=0$ or $x=1$, but the former is impossible since $1 \leq 0 \proves \bot$.
	We must show that the map sending $f \in \cC(X,\two\cdot\one)$ to $f^{-1}(1) \in \CU(X)$ is an order-isomorphism. 
	
	To see that it is an order-embedding, consider $f, g \colon X \to \two\cdot\one$. We claim that if $f(x)=1 \proves g(x)=1$, then $f \leq g$. There are four possible cases, depending on whether $f(x)=0$ or $f(x)=1$, and whether $g(x)=0$ or $g(x)=1$. The only case in which $f(x)\leq g(x)$ does not hold is when $f(x)=1$ and $g(x)=0$. However, the latter cannot occur because $f(x)=1 \proves g(x)=1$.
	
	For surjectivity, let $U \in \CU(X)$ and let $D$ be its complement in $\Sub(X)$. We define a relation $R\subseteq X\times \two\cdot\one$ by
	\[ R(x,y) \iff (U(x) \land y=1) \lor (D(x) \land y=0) . \]
	We have $\proves \exists y \qcolon R(x,y)$, since either $U(x)$ and we take $y=1$, or $D(x)$ and we take $y=0$. To show that $R$ is also order-preserving, suppose that $x \leq x'$, $R(x,y)$ and $R(x',y')$. We need to deduce that $y \leq y'$. By case analysis on $(y,y')$, it is enough to note that $(y=1)\land(y'=0)$ is impossible. In fact, it would mean that $x\leq x' \land R(x,1) \land R(x',0)$. But $R(x,1)$ implies $U(x)$, and $R(x',0)$ implies $D(x')$, contradicting the fact that $U$ is upward closed. By Proposition~\ref{p:functional-rels}, there is a morphism $f \colon X \to \two\cdot\one$ whose graph is~$R$. Since $R(x,1)$ if and only if $U(x)$, $f$ satisfies $f^{-1}(1) = U$.
\end{proof}

\begin{lemma}\label{lem:bivalued}
	If $\one$ is a generator of $\cC$, then its only subobjects are $\bot$ and $\top$.
\end{lemma}

\begin{proof}
	If $\one$ is a generator of $\cC$, then every object with no points is the empty coproduct, and so it is empty. Thus, any non-empty subobject $A \subseteq \one$ has a point. The composite $\one \to A \into \one$ is surjective because it is the identity, and therefore $A = \one$.
\end{proof}

\begin{lemma}\label{lem:compl-up-tensor}
	If $\cC$ is non-degenerate and $\one$ is a generator, then $\CU(P\cdot\one)\cong \Pos(P,\two)$ for each poset $P$, naturally in $P$.
\end{lemma}

\begin{proof}
By Lemma~\ref{lem:clopen-up-rep}, it suffices to show that $\cC(P\cdot\one,\two\cdot\one) \cong \Pos(P,\two)$. Suppose for a moment that the statement holds when $P$ is the one-element poset, that is $\cC(\one,\two\cdot\one) \cong \two$. Then, for every poset $P$, we have
\[
\cC(P\cdot\one,\two\cdot\one) \cong \Pos(P,\C(\one,\two\cdot\one)) \cong \Pos(P,\two).
\]

It remains to prove that $\cC(\one,\two\cdot\one) \cong \two$, i.e., the constants $0, 1 \colon \one \to \two\cdot\one$ are distinct and for any $c \colon \one \to \two\cdot\one$, either $\proves c=0$ or $\proves c=1$. The constants $0,1$ are distinct since $\cC$ is non-degenerate. Moreover, we know that $\proves (x=0) \lor (x=1)$ and $(x=0) \land (x=1) \proves \bot$, so this holds in particular for $x=c$. By Lemma~\ref{lem:bivalued}, the only pairs of complements in $\Sub(\one)$ are $(\top,\bot)$ and $(\bot,\top)$, and thus either $\proves c=0$ or $\proves c=1$.
\end{proof}

We are now ready to define the desired functor into $\KOrd$.
Suppose that $\cC$ is a non-degenerate pretopos whose terminal object is a generator. By composing the functor $\CU \colon \cC^\op \to \DL$ with the Priestley duality functor $\DL^\op \to \KOrd$ (which is poset-enriched, and whose essential image consists of Priestley spaces~\cite{Priestley1970}), we obtain a functor $\cC \to \KOrd$. By Lemma \ref{lem:compl-up-tensor}, the latter restricts to a functor $\xi$
\begin{equation}\label{eq:functor-xi} 
\begin{tikzcd}[row sep=1.5em]
{\cC} \arrow{r} & {\KOrd} \\
{\setst{S\cdot\one}{S\in\Set}} \arrow[dashed]{r}{\xi} \arrow[hookrightarrow]{u} & {\setst{\beta(S)}{S\in\Set}} \arrow[hookrightarrow]{u}
\end{tikzcd}
\end{equation}
such that $\xi(S\cdot\one) \cong \beta(S)$ naturally in $S$. Here, $\beta(S)$ is the Stone--\v{C}ech compactification of the discrete space $S$, whose elements are the ultrafilters of $\P(S)$.

\section{Order-filtrality, and a characterisation of $\KOrd$}\label{s:char-KOrd}

In this section, we provide a proof of the characterisation of the category $\KOrd$ stated in Theorem~\ref{thm:main-KOrd}. In more detail, we will study conditions ensuring that the functor $\xi$ in eq.~\eqref{eq:functor-xi} is an equivalence, and apply Corollary~\ref{c:proj-cover-equi} to deduce that $\cC \simeq \KOrd$.

For any set $S$, let $\xi_S \colon \cC(\one,S\cdot\one) \to \beta(S)$ be the action of $\xi$ on the appropriate posets of points, where we identify $\beta(S)$ with its set of points.
	Explicitly, $\xi_S$ sends a point $p \colon \one \to S\cdot\one$ to the ultrafilter of all the $T \subseteq S$ such that $p$ is in the inverse image of $\set{1} \subseteq \two\cdot\one$ under the associated arrow $S\cdot\one \to \two\cdot\one$ (recall from Lemmas~\ref{lem:clopen-up-rep} and~\ref{lem:compl-up-tensor} that $\cC(S\cdot\one, \two\cdot\one)\cong \P(S)$). 
	This inverse image, denoted by $\o{T}$, is an upward complemented subobject of $S\cdot\one$. It is possible to show that $\o{T}$ is isomorphic to $T\cdot\one$ with the canonical arrow to $S\cdot\one$, but we will not use this fact.

\begin{remark}
In fact, $\xi_{(-)}$ is the \emph{unique} natural transformation $\cC(\one,-\cdot\one) \to \beta$, because $\beta$ is terminal among the endofunctors on $\Set$ that preserve finite coproducts. This result is due to B{\"o}rger~\cite{Borger1987} and was recently extended by Garner in~\cite{GarUFC2020}. B{\"o}rger's theorem was used by Richter in~\cite{Richter1991} to characterise $\beta$ as the unique monad $T$ on $\Set$ that preserves finite coproducts and satisfies separation and compactness properties. These properties correspond, respectively, to the injectivity and surjectivity of $T \to \beta$. We will adapt this idea to the context of compact ordered spaces.
\end{remark}

\begin{lemma}\label{lem:S-equiv}
	$\xi$ is an equivalence if, and only if, $\xi_S \colon \cC(\one,S\cdot\one) \to \beta(S)$ is a bijection for all sets $S$.
\end{lemma}

\begin{proof}
	Note that $\xi$ is an equivalence if, and only if, all the induced maps $\cC(T\cdot\one,S\cdot\one) \to \KOrd(\beta(T),\beta(S))$ are bijective, since the order on the codomain is discrete.
	
	By naturality of $\xi(T\cdot\one) \cong \beta(T)$ in $T$, the transformation 
	\[
	\cC(T\cdot\one,S\cdot\one) \to \KOrd(\beta(T),\beta(S))
	\] 
	is natural in $T$ and $S$. In particular, diagram~\eqref{diag:S-equiv-A} commutes for all $x \in T$, where the horizontal arrows $p_x$ and $q_x$ act by precomposing with the image of $x \colon 1 \to T$ under the functors $-\cdot\one$ and $\beta$, respectively. Consequently, diagram~\eqref{diag:S-equiv-B} also commutes.\\
	\begin{minipage}{0.5\textwidth}
	\begin{equation}\label{diag:S-equiv-A}\begin{tikzcd}
		\cC(T\cdot\one,S\cdot\one) \ar[r,"p_x"] \ar[d] &[-1em] \cC(\one,S\cdot\one) \ar[d]\\
		\KOrd(\beta(T),\beta(S)) \ar[r,"q_x"] & \beta(S)
	\end{tikzcd}\end{equation}
	\end{minipage}\hfill
	\begin{minipage}{0.5\textwidth}
	\begin{equation}\label{diag:S-equiv-B}\begin{tikzcd}
		\cC(T\cdot\one,S\cdot\one) \ar[r,phantom,"\cong"{description,sloped}] \ar[d] &[-2em] \cC(\one,S\cdot\one)^T \ar[d]\\
		\KOrd(\beta(T),\beta(S)) \ar[r,phantom,"\cong"{description,sloped}] & \beta(S)^T
	\end{tikzcd}\end{equation}
	\end{minipage}
	
	\vskip.2\baselineskip \noindent Hence, if $\xi_{S}\colon \cC(\one,S\cdot\one) \to \beta(S)$ is bijective, so is $\cC(T\cdot\one,S\cdot\one) \to \KOrd(\beta(T),\beta(S))$.
\end{proof}

To identify a sufficient condition ensuring that the maps $\xi_{S}$ are bijective, we introduce the notion of order-filtrality, extending the concept of filtrality from~\cite{mr} to the poset-enriched setting. Given an object $X$ of a coherent category, let $\Up(X)$ be the sub-poset of $\Sub(X)$ consisting of the upward subobjects. It is not difficult to see that $\Up(X)$ is a sublattice of $\Sub(X)$. If $L$ is a distributive lattice, we denote by $\Filt(L)$ the set of non-empty filters of $L$, partially ordered by reverse inclusion.

\begin{definition}\label{d:order-filtral}
	An object $X$ of a coherent category is \emph{order-filtral} if the order-preserving map
	\[
	\Up(X) \to \Filt(\CU(X)), \ \ U\mapsto \{V\in \CU(X)\mid U\subseteq V\}
	\]
	is an order-isomorphism.
\end{definition}

\begin{example}
Order-filtral objects in $\Pos$ are finite posets. In $\KOrd$, they coincide with Priestley spaces, i.e., \emph{totally order-disconnected} compact ordered spaces~\cite{Priestley1970}.
\end{example}

We show that, in a well-pointed cocomplete pretopos, order-filtral objects are compact and separated in an appropriate sense defined below. 
Note that in a cocomplete pretopos, each poset of subobjects $\Sub(X)$ is complete. This implies that $\Up(X)$ is also complete, since it is a reflective sub-poset of $\Sub(X)$; the reflector sends $Y \subseteq X$ to $\up{Y}$.

\begin{definition}
An object $X$ of a cocomplete pretopos $\C$ is \emph{compact} if, for every codirected set $\Gamma\subseteq \CU(X)$ of upward complemented subobjects of $X$, 
\[
\bigwedge \Gamma = 0 \ \text{ in $\Up(X)$ } \Longrightarrow \ 0\in \Gamma.
\]
Moreover, $X$ is \emph{separated} if for any two distinct points $p,q\colon \one\to X$, there exists an upward complemented subobject $U\in \CU(X)$ containing exactly one of $p$ or $q$.
\end{definition}

\begin{proposition}\label{p:filtral-iff-compact-sep}
	Let $X$ be an object of a well-pointed cocomplete pretopos. If $X$ is order-filtral, then it is compact and separated.
\end{proposition}

\begin{proof}
	Since compactness is a property of $\Up(X)$, it is enough to observe that for any distributive lattice $L$, and every codirected subset $\Gamma\subseteq \Filt(L)$, if $\bigwedge \Gamma = 0$ then $0\in \Gamma$. 
	
	To prove that $X$ is separated, pick distinct points $p, q \colon \one \to X$. The subobjects ${\uparrow}p$ and ${\uparrow}q$ are different, for otherwise $p \leq q$ and $q \leq p$, which implies $p=q$. Hence, there is an upward complemented subobject of $X$ containing one of ${\uparrow}p$ and ${\uparrow}q$, but not both. This shows that $X$ is separated.
\end{proof}

Compact objects, and separated ones, have the following useful closure properties.

\begin{lemma}\label{l:properties-compact-sep}
The following statements hold for every arrow $f\colon X\to Y$ in a well-pointed cocomplete pretopos:
\begin{enumerate}[label=(\alph*)]
\item\label{i:image-compact} If $f$ is a surjection and $X$ is compact, then $Y$ is compact;
\item\label{i:sub-separated} If $f$ is an injection and $Y$ is separated, then $X$ is separated.
\end{enumerate}
\end{lemma}

\begin{proof}
\ref{i:image-compact} Suppose that $X$ is compact and $f$ is a surjection. Then $f[f^{-1}(A)] = A$ so that $f^{-1} \colon \Sub(Y) \to \Sub(X)$ is injective. Moreover, $f^{-1}$ preserves upward complemented subobjects as well as arbitrary infima, since it has a left adjoint $f[-]$. Hence, if $\Gamma \subseteq \CU(Y)$ is codirected and $\bigwedge \Gamma = 0$ in $\Up(Y)$, then 
\[
\bigwedge \setst{f^{-1}(\gamma)}{\gamma\in\Gamma} = f^{-1}\big(\bigwedge \Gamma\big) = 0
\]
in $\Up(X)$.
Since $X$ is compact, there exists $\gamma\in\Gamma$ with $f^{-1}(\gamma) = 0$, and so $\gamma=0$ because $f^{-1}$ is injective. This shows that $Y$ is compact.
	
\ref{i:sub-separated} Assume that $Y$ is separated and $f$ is an injection. If $p, q \colon \one \to X$ are distinct points, then $f(p) \neq f(q)$ because $f$ is injective. Since $Y$ is separated, there is $U \in \CU(Y)$ containing exactly one of $f(p)$ or $f(q)$. It follows that $f^{-1}(U) \in \CU(X)$ contains exactly one of $p$ or $q$. Therefore, $X$ is separated.
\end{proof}

Finally, we prove that compactness and separation force the maps $\xi_{S}$ to be bijections.

\begin{lemma}\label{l:compact-sep-iso}
For any set $S$, if the copower $S\cdot\one$ is compact and separated, then $\xi_{S}\colon \cC(\one,S\cdot \one) \to \beta(S)$ is a bijection.
\end{lemma}

\begin{proof}
	We claim that if $S\cdot\one$ is compact, then $\xi_S$ is surjective. Let $\mu \in \beta(S)$ be an ultrafilter. For any $T \in \mu$, the upward complemented subobject $\o{T} \subseteq S\cdot\one$ is non-empty since it contains at least a point coming from $T$. By compactness of $S\cdot\one$, the intersection of all the $\o{T}$ for $T \in \mu$ is non-empty, so it has a point, whose image under $\xi_S$ is $\mu$.
	
	Next, we claim that if $S\cdot\one$ is separated, then $\xi_S$ is injective. Given distinct points $p, q \colon \one \to S\cdot\one$, there is an upward complemented subobject $U\subseteq S\cdot\one$ containing, without loss of generality, $p$ and not $q$. By Lemma~\ref{lem:compl-up-tensor}, $U$ is of the form $\o{T}$ for some $T \subseteq S$. Such a $T$ is in $\xi_S(p)$ but not in $\xi_S(q)$, therefore $\xi_S(p) \neq \xi_S(q)$.
\end{proof}

We are now in a position to characterise $\KOrd$ among the poset-enriched categories. As previously mentioned, $\KOrd$ is a (non-degenerate) pretopos. Moreover, the $S$-indexed copower of its terminal object is $\beta(S)$, which is order-filtral, and each compact ordered space is covered by the Stone--\v{C}ech compactification of a set. Thus, $\KOrd$ satisfies the properties in Theorem~\ref{thm:main-KOrd}. We show that any category that satisfies these properties is equivalent to $\KOrd$, thus concluding the proof of the latter theorem.
\begin{theorem}\label{th:KOrd-characterisation}
Let $\C$ be a non-degenerate pretopos satisfying the following properties:
\begin{enumerate}[label=(\roman*)]
\item\label{i:regular-gen} The terminal object $\one$ of $\C$ is a discrete generator;
\item Every object of $\C$ is covered by an order-filtral object.
\end{enumerate}
Then $\C$ is equivalent to $\KOrd$. 
\end{theorem}

\begin{proof}
Recall from Proposition~\ref{p:gen-iff-dicrete-gen-and-coinserters} that, in an exact category, an object is a generator if and only if it is a discrete generator. By Lemmas~\ref{l:projective-discrete-gen-cover} and~\ref{l:one-proj}, $\setst{S\cdot\one}{S\in\Set}$ is a projective cover of $\cC$, and $\setst{\beta(S)}{S\in\Set}$ is a projective cover of $\KOrd$. In view of Corollary~\ref{c:proj-cover-equi}, it is thus sufficient to show that the functor
\[ 
\xi \colon \setst{S\cdot\one}{S\in\Set} \to \setst{\beta(S)}{S\in\Set}
\]
from eq.~\eqref{eq:functor-xi} is an equivalence. By Lemmas~\ref{lem:S-equiv} and~\ref{l:compact-sep-iso}, this holds if the copower $S\cdot \one$ is compact and separated for each set $S$.

	Let $f\colon X\epi S\cdot \one$ be a surjection whose domain is order-filtral. The pretopos $\cC$ is well-pointed and cocomplete by Proposition~\ref{prop:well-pointed-cocomplete}, hence $X$ is compact and separated by Proposition~\ref{p:filtral-iff-compact-sep}. It follows from Lemma~\ref{l:properties-compact-sep}\ref{i:image-compact} that $S\cdot\one$ is compact. Moreover, since $S\cdot\one$ is projective, there exists $s\colon S\cdot\one \to X$ such that $sf = \id$. Such a split mono $s$ is an injection, thus $S\cdot \one$ is separated by Lemma~\ref{l:properties-compact-sep}\ref{i:sub-separated}.
\end{proof}

\begin{remark}
Direct inspection shows that Theorem~\ref{thm:main-KOrd} can be slightly strengthened by replacing order-filtral objects with compact separated ones.
\end{remark}


\Urlmuskip=0mu plus 1mu\relax
{\printbibliography}

\end{document}